\documentclass{article}
\usepackage{amsmath}
\usepackage{amsfonts}

\newtheorem{theorem}{Theorem}

\newtheorem{conclusion}{Conclusion}

\newtheorem{example}{Example}

\newenvironment{proof}[1][Proof]{\noindent\textbf{#1.} }{\ \rule{0.5em}{0.5em}}
\input{tcilatex}
\begin{document}

\title{Viviani's Theorem and its Extensions Revisited;\\
Canghareeb, Sanghareeb and a New Definition of the Ellipse}
\author{}
\maketitle

A two dimensional outsider called "\textit{Canghareeb}"\footnote{%
The word "Canghareeb" comes from Arabic, and is an abbreviation of two words
"caen ghareeb" which means outsider.}, consisting of a head and three legs,
lives inside a triangle under the following conditions;

\begin{itemize}
\item For better balance and movement he puts its legs on the sides of the
triangle perpendicularly.

\item He can extend one leg or two legs at the expense of the others so that 
\textbf{the sum of the lengths of the three legs stays constant}.

\item He can not bend its legs.
\end{itemize}

The following questions arise; Where does Canghareeb live inside the
triangle? Are there necessary conditions for Canghareeb to live inside a
given triangle? What happens if the triangle is equilateral, isosceles or
scalene?

In Fig. \ref{fig1}, Canghareeb is illustrated by a point (head) and three
perpendiculars (legs). The questions are intended to describe the set of
points formed by the traces of the head when Canghareeb moves inside the
triangle.

These questions are related to Viviani's theorem and its extensions (see 
\cite{Abb}), as will be illustrated.

\section*{Canghareeb}

We shall rely on the following basic theorem (see \cite{Abb} Theorems 1 and
2) to conclude the results in this section. At first recall the following
terminology: Let $\mathcal{P}$ be a polygon consisting of both boundary and
interior points. Define the \textit{distance sum function} $\mathcal{V}:%
\mathcal{P}\rightarrow \mathbb{R},$ where for each point $P\in \mathcal{P}$ $%
\ $the value $\mathcal{V}(P)$ is defined as the sum of the distances from $P$
to the sides of $\mathcal{P}$.

\begin{theorem}
\label{equilateral}Any triangle can be divided into parallel line segments
on which $\mathcal{V}$ is constant. Furthermore, the following conditions
are equivalent:

\begin{itemize}
\item $\mathcal{V}$ is constant on\textit{\ }the\textit{\ triangle }$%
\triangle ABC.$

\item \textit{There are three non-collinear points, inside the triangle, at
which }$\mathcal{V}$ takes the same value\textit{.}

\item \textit{\ }$\triangle ABC$ is \textit{equilateral. }
\end{itemize}
\end{theorem}

\begin{theorem}
\label{convex polygon}(a) Any convex polygon $\mathcal{P}$ can be divided
into parallel line segments, on which $\mathcal{V}$ is constant. \ 

(b) If $\ \mathcal{V}$ takes equal values at three non-collinear points,
inside a convex polygon, then $\mathcal{V}$ is constant on $\mathcal{P}$ .
\end{theorem}

\subsection*{Equilateral triangle}

By Viviani's theorem; the sum of the distances from the sides of any point
inside an equilateral triangle is constant. The constant sum is the height
of the equilateral triangle. Thus we have the following conclusion:

\begin{conclusion}
If the total length of the legs of Canghareeb equals the height of an
equilateral triangle then Canghareeb can live inside the triangle and move
freely from one point inside the triangle to another.
\end{conclusion}

\subsection*{Isosceles triangle}

Since an isosceles triangle has a reflection symmetry across the height we
conclude that if Canghareeb can reach a point $P$ inside the triangle then
he can reach the reflection point $P^{\prime }$ across the height. Thus we
have:

\begin{conclusion}
Canghareeb can live on a line segment parallel to the base of an isosceles
triangle provided the total length of its legs ranges between the lengths of
the smallest and the largest altitudes of the triangle.
\end{conclusion}

\begin{proof}
If Canghareeb can reach point $P$ inside the triangle then he can reach the
reflection point $P^{\prime }$ across the height. Indeed, In Fig. \ref{fig2}%
, if the line segment $DF$ is parallel to the base $BC$ then when $P$ moves
along $DF$, the length $a$ remains constant and $b+c=DP\sin \alpha +PE\sin
\alpha =DE\sin \alpha $ which is also constant. By Theorem \ref{equilateral}%
, Canghareeb can not live at any other point inside the triangle unless it
is equilateral.
\end{proof}

\subsection*{Scalene triangle}

Similarly, by Theorem \ref{convex polygon}, Canghareeb can live inside a
scalene triangle, provided the following necessary condition is satisfied:
the total length of its legs ranges between the lengths of the smallest and
the largest altitudes of the triangle. This is true due to the fact that the
distance sum function $\mathcal{V}$ is a linear continuous function in two
variables hence; \textbf{it takes\ on every value between its minimum and
its maximum}.

Thus we have:

\begin{conclusion}
Canghareeb can live on a line segment inside a scalene triangle provided the
total length of its legs ranges between the lengths of the smallest and the
largest altitudes of the triangle.
\end{conclusion}

The question is how can we determine this segment for a general triangle?

We shall illustrate the method explained in \cite{Abb} by the following
example:

\begin{example}
Let ABC be the right angled triangle with vertices $(0,0),(0,3),(4,0)$
respectively (see Fig. \ref{fig3}). If the total length of Canghareeb's legs
is $l,$ $3\leq l\leq $ $4,$ then Canghareeb can live on a line segment
inside the triangle parallel to the line $2x+y=0$.

In Fig. \ref{fig3}, $l=3.16743$ and the living area is the line segment $HI.$
\end{example}

\begin{proof}
The smallest altitude of the triangle is $3$ and the largest altitude is $4.$
Hence by the previous conclusion, Canghareeb can live inside the triangle
across a line segment. To find this line segment, we compute the distance
sum function $\mathcal{V};$ The equation of the hypotenuse is $3x+4y=12.$
Hence,

\begin{equation*}
\mathcal{V}=x+y-\frac{3x+4y-12}{5}=\frac{2}{5}x+\frac{1}{5}y-\frac{12}{5}.
\end{equation*}

Therefore, the lines $\mathcal{V}=c$ are parallel to the line $2x+y=0$ and
the result follows.
\end{proof}

\ The following example will be left to the reader but at the end of the
paper a clue will be given.

\begin{example}
\label{Ex4}One Canghareeb proclaimed that he can reach the three points $%
(-2,0),(\frac{1}{2},\frac{1}{3}),(2,0)$ inside a triangle. Show that the
triangle is equilateral. If the coordinates of two of the vertices are $%
(-2,0),(2,0)$ find the coordinates of the third vertex of the triangle.
\end{example}

\section*{Sanghareeb}

Motivated by Canghareeb's story, another mysterious creature called \textit{%
Sanghareeb} lives in the plane relative to a triangle with the following
conditions;

\begin{itemize}
\item He puts its legs on the sides of the triangle or their extensions
perpendicularly.

\item He can extend one leg or two legs at the expense of the others so that 
\textbf{the} \textbf{sum of the squares of the lengths of the three legs
stays constant}.

\item He can not bend its legs.
\end{itemize}

One day Sanghareeb met Canghareeb inside the triangle $ABC$ shown in Fig. %
\ref{fig3}, with vertices $(0,0),(0,3),(4,0)$. The total length of
Canghareeb's legs was $3.16743$ and the sum of the squares of the lengths of
Sanghareeb's legs was $5$. Sanghareeb boasted to Canghareeb that he can
reach points outside the triangle. When the latter asked about the former's
secret, Sanghareeb showed Canghareeb the journey he could make. In fact
Sanghareeb walked on a closed smooth curve, called the \textit{ellipse}. He
made the following mathematics;

Referring to Fig. \ref{fig3}, $QR^{2}+QS^{2}+QT^{2}=5$ implies the following
equation of a quadratic curve;

\begin{equation*}
x^{2}+y^{2}+\left( \frac{3x+4y-12}{5}\right) ^{2}=5.
\end{equation*}

Simplifying, one gets the equivalent equation;

\begin{equation*}
34x^{2}+41y^{2}+24xy-72x-96y+19=0.
\end{equation*}

This is exactly the equation of the ellipse shown in Fig. \ref{fig3} and
Sanghareeb could meet Canghareeb at exactly two points $L$ and $J$ inside
the triangle.

\section*{A new definition of the ellipse}

In view of the previous discussion we may state the following theorem, which
gives a new definition of the ellipse;

\begin{theorem}
An ellipse is the locus of points which have a constant sum of squares of
distances from the sides of a given triangle.
\end{theorem}

\begin{proof}
We shall prove the following two claims;

(a) Given a triangle, the locus of points which have a constant sum of
squares of distances from the sides is an ellipse.

(b) Given an ellipse, there is a triangle for which the sum of the squares
of the distances from the sides for all points on the ellipse is constant.

Using analytic geometry, we choose for part (a) the coordinate system such
that the vertices of the triangle lie on the axes.

Suppose the coordinates of the vertices of the triangle are $%
A(0,a),B(-b,0),C(c,0),$ where $a,b,c>0$ (see Fig. \ref{fig4}). Let $(x,y)$
be any point in the plane and let $d_{1},d_{2},d_{3}$ be the distances of $%
(x,y)$ from the sides of the triangle $ABC$. Clearly we have

\begin{equation*}
\underset{i=1}{\overset{3}{\sum }}d_{i}^{2}=\frac{(ax+cy-ac)^{2}}{a^{2}+c^{2}%
}+\frac{(ax-by+ab)^{2}}{a^{2}+b^{2}}+y^{2}.
\end{equation*}

Hence, $\underset{i=1}{\overset{3}{\sum }}d_{i}^{2}=k$ (constant) if and
only if the point $(x,y)$ lies on the quadratic curve 
\begin{equation}
\frac{(ax+cy-ac)^{2}}{a^{2}+c^{2}}+\frac{(ax-by+ab)^{2}}{a^{2}+b^{2}}%
+y^{2}=k.  \label{eq1}
\end{equation}

In general, a quadratic equation in two variables;

\begin{equation*}
Ax^{2}+Bxy+Cy^{2}+Dx+Ey+F=0,
\end{equation*}

represents an ellipse provided that the discriminant%
\begin{equation*}
\Delta =B^{2}-4AC<0.
\end{equation*}

It is easily seen that this is the case for equation (\ref{eq1}).

For part (b), after rotation or translation of the axes we may assume that
the ellipse has a canonical form; $\frac{x^{2}}{\alpha ^{2}}+\frac{y^{2}}{%
\beta ^{2}}=1,$ $\alpha \geq \beta .$

Because of the symmetry of the ellipse, we shall take an isosceles triangle
with vertices $A(0,a),B(-b,0),C(b,0)$ and compute the sum of the squares of
the distances from its sides. Substituting b=c in equation \ref{eq1}, we have

\begin{equation*}
\frac{(ax+by-ab)^{2}}{a^{2}+b^{2}}+\frac{(ax-by+ab)^{2}}{a^{2}+b^{2}}%
+y^{2}=k.
\end{equation*}

Equivalently, we get a translation of a canonical ellipse;

\begin{equation*}
\frac{x^{2}}{a^{2}+3b^{2}}+\frac{(y-\frac{2ab^{2}}{a^{2}+3b^{2}})^{2}}{2a^{2}%
}=\frac{(a^{2}+b^{2})k-2a^{2}b^{2}}{2a^{2}(a^{2}+3b^{2})}+\frac{2b^{4}}{%
(a^{2}+3b^{2})^{2}}.
\end{equation*}

Thus, it is enough to take 
\begin{equation*}
\frac{(a^{2}+b^{2})k-2a^{2}b^{2}}{2a^{2}(a^{2}+3b^{2})}+\frac{2b^{4}}{%
(a^{2}+3b^{2})^{2}}=1,
\end{equation*}

and therefore,%
\begin{equation}
k=\frac{2a^{2}(a^{4}+7a^{2}b^{2}+10b^{4})}{(a^{2}+b^{2})(a^{2}+3b^{2})}.
\label{eq2}
\end{equation}

Translating downward by $\frac{2ab^{2}}{a^{2}+3b^{2}},$ we get that the
ellipse $\frac{x^{2}}{a^{2}+3b^{2}}+\frac{y^{2}}{2a^{2}}=1$ is the locus of
points which have a constant sum of squares of distances from the sides of
the triangle with vertices $A(0,a-\frac{2ab^{2}}{a^{2}+3b^{2}}),B(-b,-\frac{%
2ab^{2}}{a^{2}+3b^{2}}),C(b,-\frac{2ab^{2}}{a^{2}+3b^{2}}).$ Moreover, this
constant is given by equation \ref{eq2}.
\end{proof}

\ 

The following table exhibits some examples;

\begin{equation*}
\begin{array}{ccccccc}
a & b & \alpha ^{2}=a^{2}+3b^{2} & \beta ^{2}=2a^{2} & \text{Ellipse} & k & 
\text{Figure} \\ 
1 & 1 & 4 & 2 & \frac{x^{2}}{4}+\frac{y^{2}}{2}=1 & \frac{9}{2} & \ref{fig5}
\\ 
1 & 2 & 13 & 2 & \frac{x^{2}}{13}+\frac{y^{2}}{2}=1 & \frac{378}{65} & \ref%
{fig6} \\ 
\sqrt{3} & 1 & 6 & 6 & x^{2}+y^{2}=6 & 10 & \ref{fig7}%
\end{array}%
\end{equation*}

Notice that the locus of points in the third example is a circle. In
general, the locus of points is a circle whenever we have $\alpha ^{2}=\beta
^{2}$ or $a^{2}+3b^{2}=2a^{2}$ which is equivalent to $a=\sqrt{3}b.$ In this
case the vertices of the triangle are $A(0,\frac{2a}{3}),B(-b,-\frac{a}{3})$
and $C(b,-\frac{a}{3})$. This implies that the triangle is equilateral.
Therefore, we have the following;

\begin{conclusion}
The locus of points which have a constant sum of squares of distances from
the sides of a given triangle is a circle if and only if the triangle is
equilateral.
\end{conclusion}

\section*{Concluding remarks}

Canghareeb and Sanghareeb are described to be creatures living in the plane
relative to a triangle. We may extend the idea to any $n-$gon. In this case
the creature will be two dimensional with $n$ legs. It may be extended to
polyhedra in which case it will be three dimensional with $n$ legs where $n$
is the number of faces.

Viviani (1622-1703), who was a student and assistant of Galileo, discovered
that equilateral triangles satisfy the following property; the sum of the
distances from the sides of any point inside an equilateral triangle is
constant.

Viviani's theorem can be easily proved by using areas. Joining a point $P$
inside the triangle to its vertices divides it into three parts. The sum of
their areas will be equal to the area of the original one. Therefore, the
sum of the distances from the sides will be equal to the height of the
triangle and the theorem follows.

Kawasaki \cite[p. 213]{Kaw}, with a proof without words, used only rotations
to establish Viviani's theorem.

Samelson \cite[p. 225]{Sam} gave a proof of Viviani's theorem that uses
vectors and Chen \& Liang \cite[p. 390-391]{CL} used this vector method to
prove a converse: if inside a triangle there is a circular region in which
the sum of the distances from the sides is constant then the triangle is
equilateral. In \cite{Abb} this theorem is generalized for convex polygons
(and for polyhedra as well); If inside a convex polygon there are three
non-collinear points for which the sum of the distances from the sides is
constant, then the sum of the distances from the sides for each point inside
the polygon is constant.

This means that if Canghareeb can reach three non-collinear points inside
some triangle then the triangle is equilateral and he can reach every point
inside the triangle. This fact is the clue for solving Example \ref{Ex4}.

\ 

{\large Acknowledgement}: \textit{This work is part of a research which was
held and supported by Beit Berl College Research Fund.}

\bigskip

\bigskip

\bigskip

\FRAME{ftbphFU}{5.6706in}{3.8363in}{0pt}{\Qcb{Canghareeb with head P and
legs $a,b$ and $c$}}{\Qlb{fig1}}{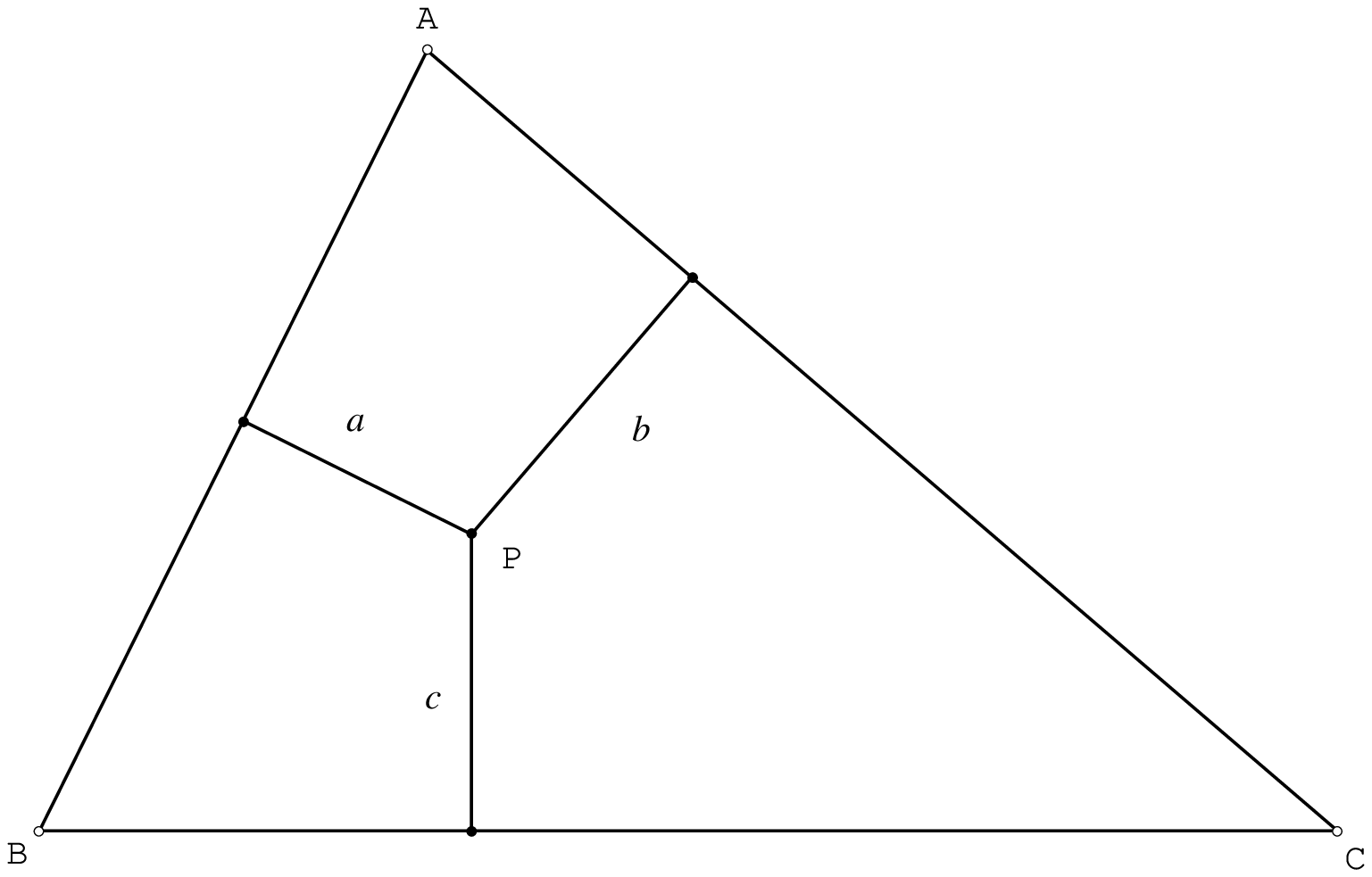}{\special{language
"Scientific Word";type "GRAPHIC";maintain-aspect-ratio TRUE;display
"USEDEF";valid_file "F";width 5.6706in;height 3.8363in;depth
0pt;original-width 10.6692in;original-height 7.2004in;cropleft "0";croptop
"1";cropright "1";cropbottom "0";filename 'cangharib.eps';file-properties
"XNPEU";}}

\FRAME{ftbpFU}{5.3592in}{3.1401in}{0pt}{\Qcb{Canghareeb lives on a segment
parallel to the base of an isosceles triangle}}{\Qlb{fig2}}{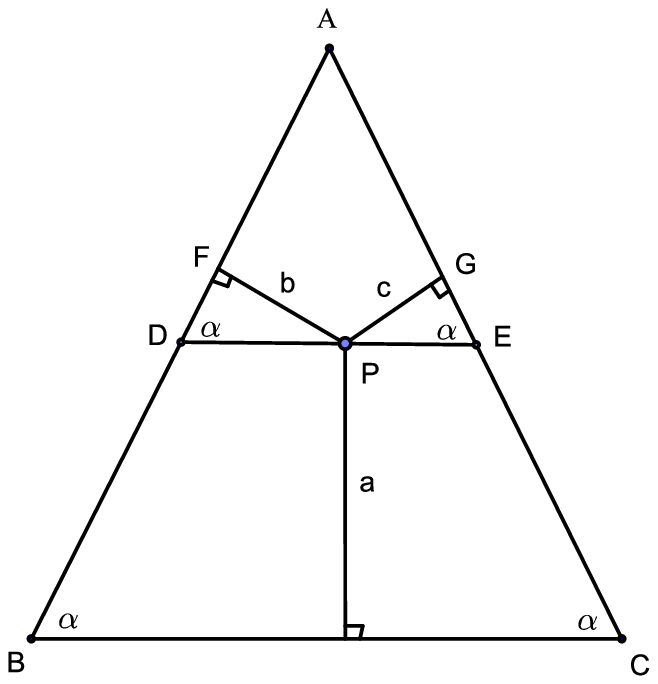}{%
\special{language "Scientific Word";type "GRAPHIC";maintain-aspect-ratio
TRUE;display "USEDEF";valid_file "F";width 5.3592in;height 3.1401in;depth
0pt;original-width 6.4013in;original-height 3.7395in;cropleft "0";croptop
"1";cropright "1";cropbottom "0";filename 'isoceles.eps';file-properties
"XNPEU";}}

\FRAME{ftbpFU}{5.8479in}{3.9548in}{0pt}{\Qcb{Canghareeb the total length of
whose legs $3.16743$, lives on the segment HI parallel to AD inside the
triangle ABC}}{\Qlb{fig3}}{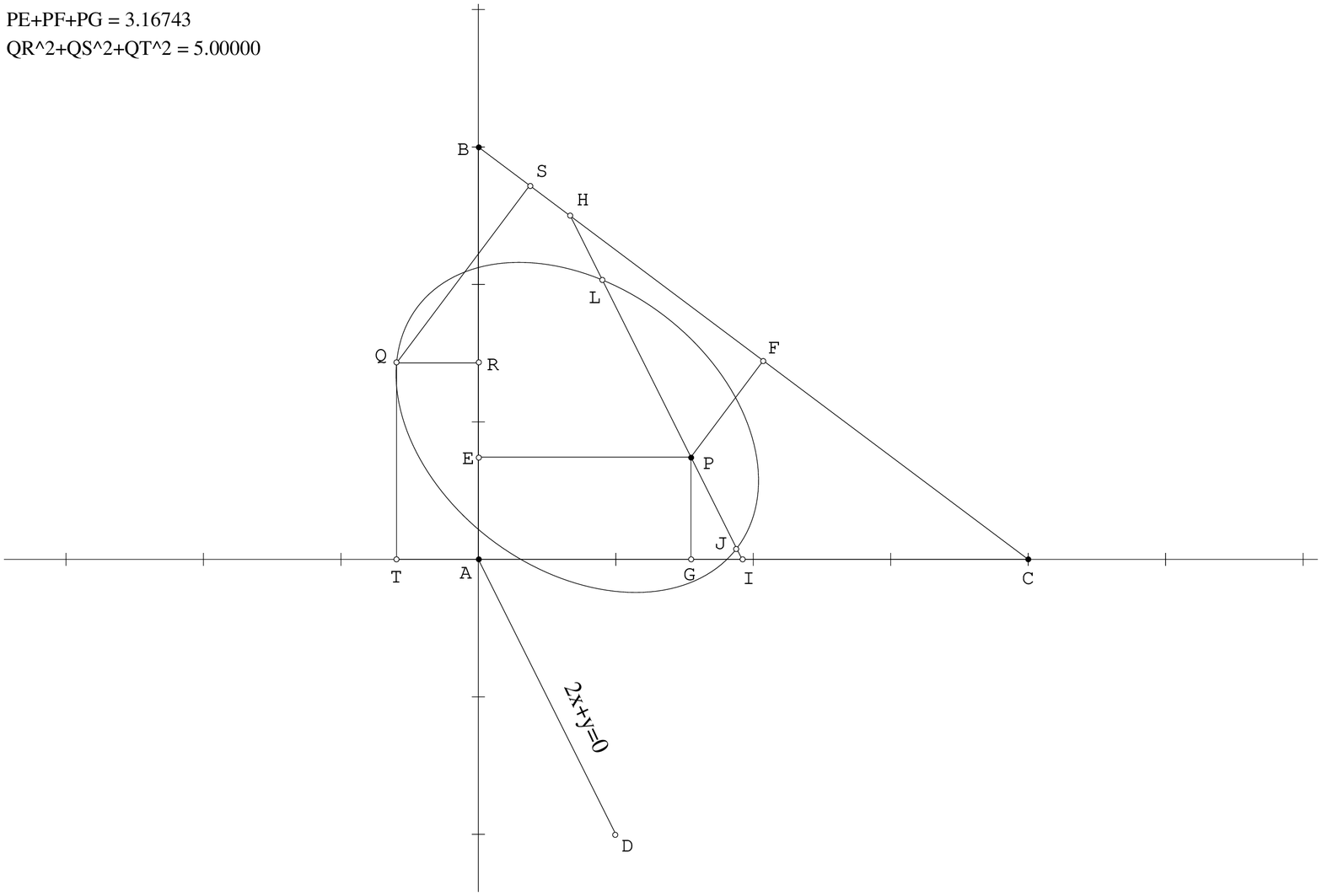}{\special{language
"Scientific Word";type "GRAPHIC";maintain-aspect-ratio TRUE;display
"USEDEF";valid_file "F";width 5.8479in;height 3.9548in;depth
0pt;original-width 10.6692in;original-height 7.2004in;cropleft "0";croptop
"1";cropright "1";cropbottom "0";filename
'rightangletriangle.eps';file-properties "XNPEU";}}

\FRAME{ftbpFU}{5.2918in}{3.096in}{0pt}{\Qcb{Sanghareeb moves along an ellipse%
}}{\Qlb{fig4}}{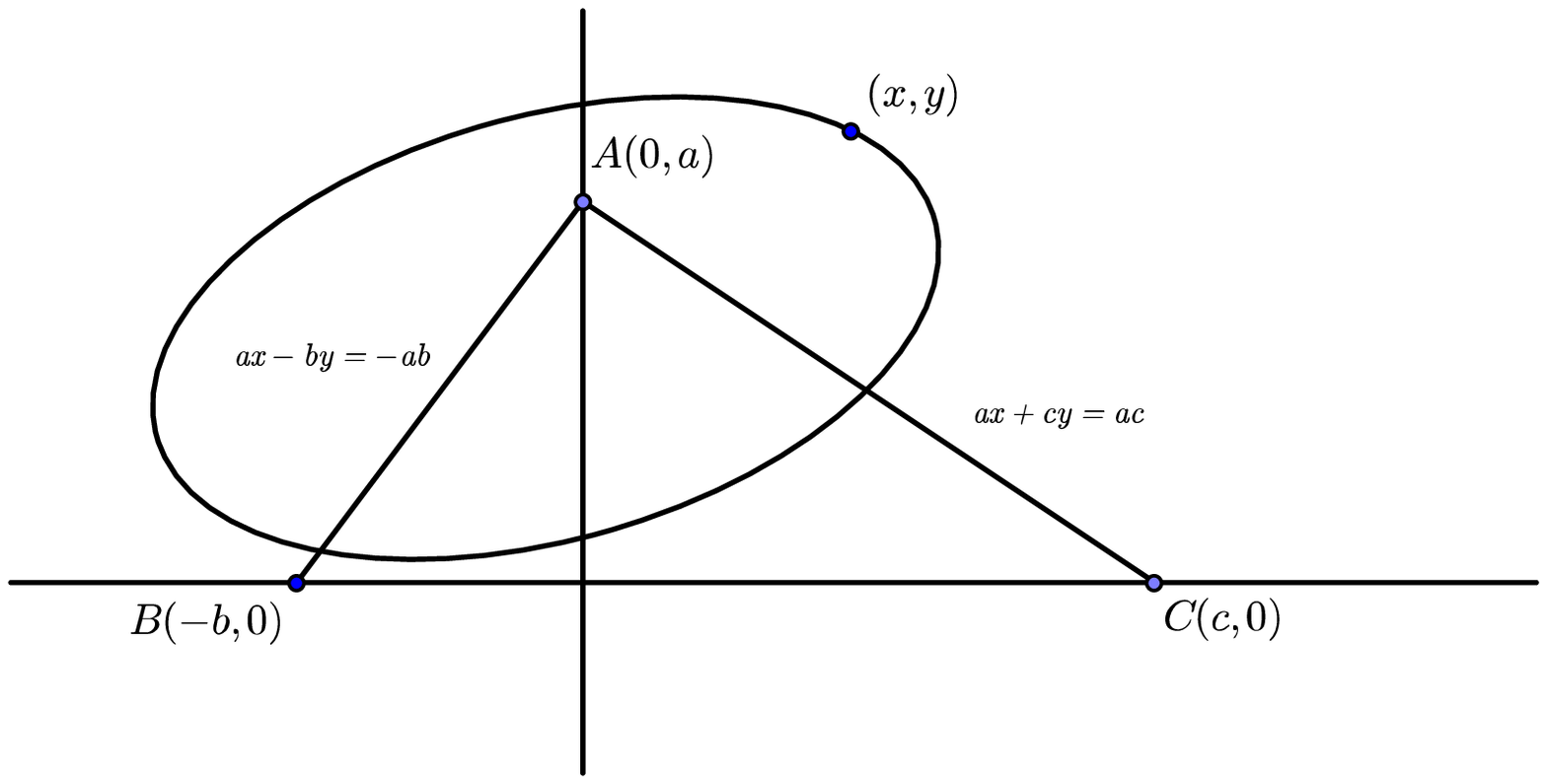}{\special{language "Scientific Word";type
"GRAPHIC";maintain-aspect-ratio TRUE;display "USEDEF";valid_file "F";width
5.2918in;height 3.096in;depth 0pt;original-width 8.0652in;original-height
4.7063in;cropleft "0";croptop "1";cropright "1";cropbottom "0";filename
'ellipse1.eps';file-properties "XNPEU";}}

\FRAME{ftbpFU}{5.7709in}{3.9038in}{0pt}{\Qcb{Sanghareeb $G$ with legs $%
GH,GI,GJ$ moves along an ellipse with $k=4.5$}}{\Qlb{fig5}}{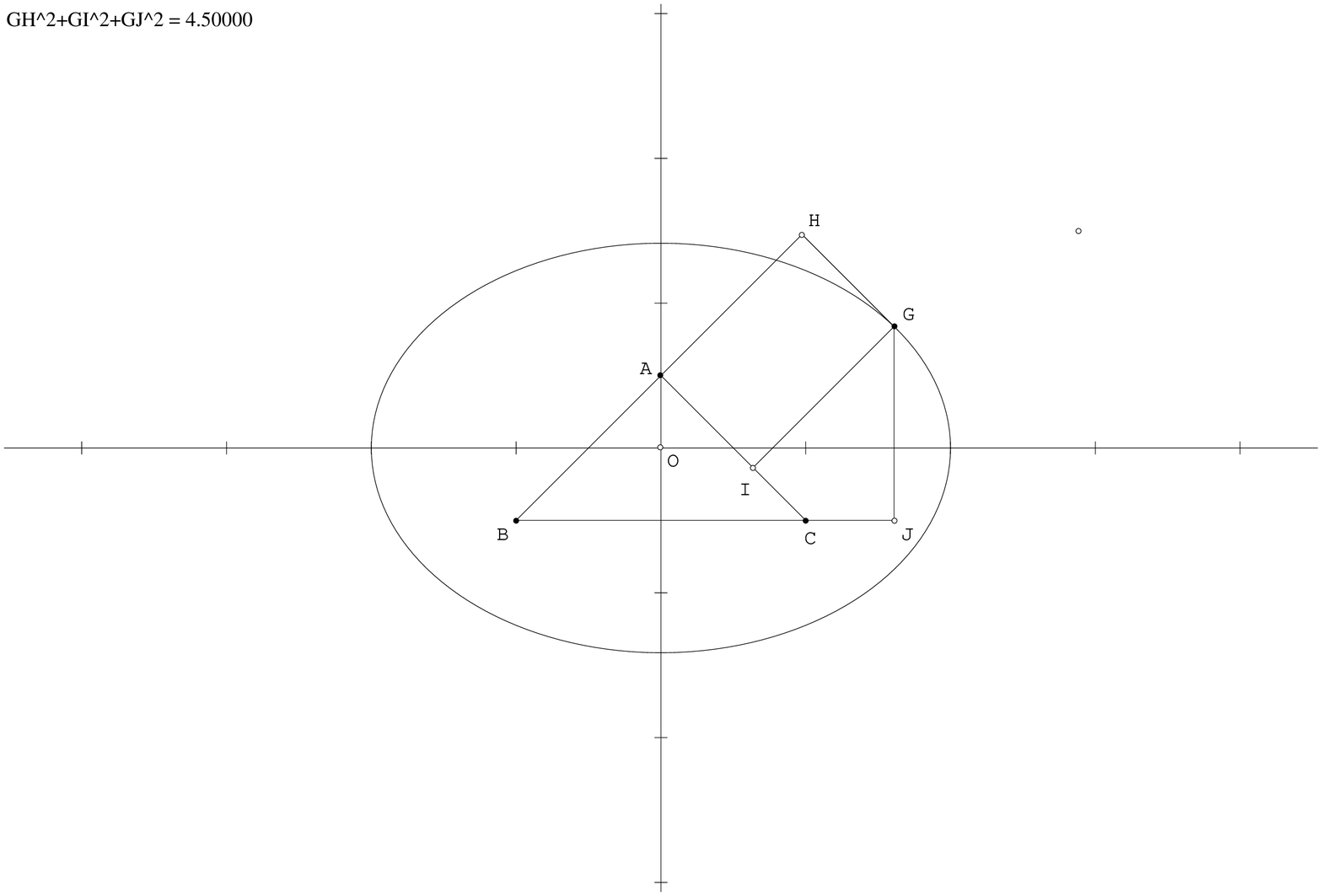}{\special%
{language "Scientific Word";type "GRAPHIC";maintain-aspect-ratio
TRUE;display "USEDEF";valid_file "F";width 5.7709in;height 3.9038in;depth
0pt;original-width 10.6692in;original-height 7.2004in;cropleft "0";croptop
"1";cropright "1";cropbottom "0";filename '1-1.eps';file-properties "XNPEU";}%
}

\FRAME{ftbpFU}{5.8435in}{3.9539in}{0pt}{\Qcb{Sanghareeb $G$ moves along an
ellipse with $k=\frac{378}{65}$}}{\Qlb{fig6}}{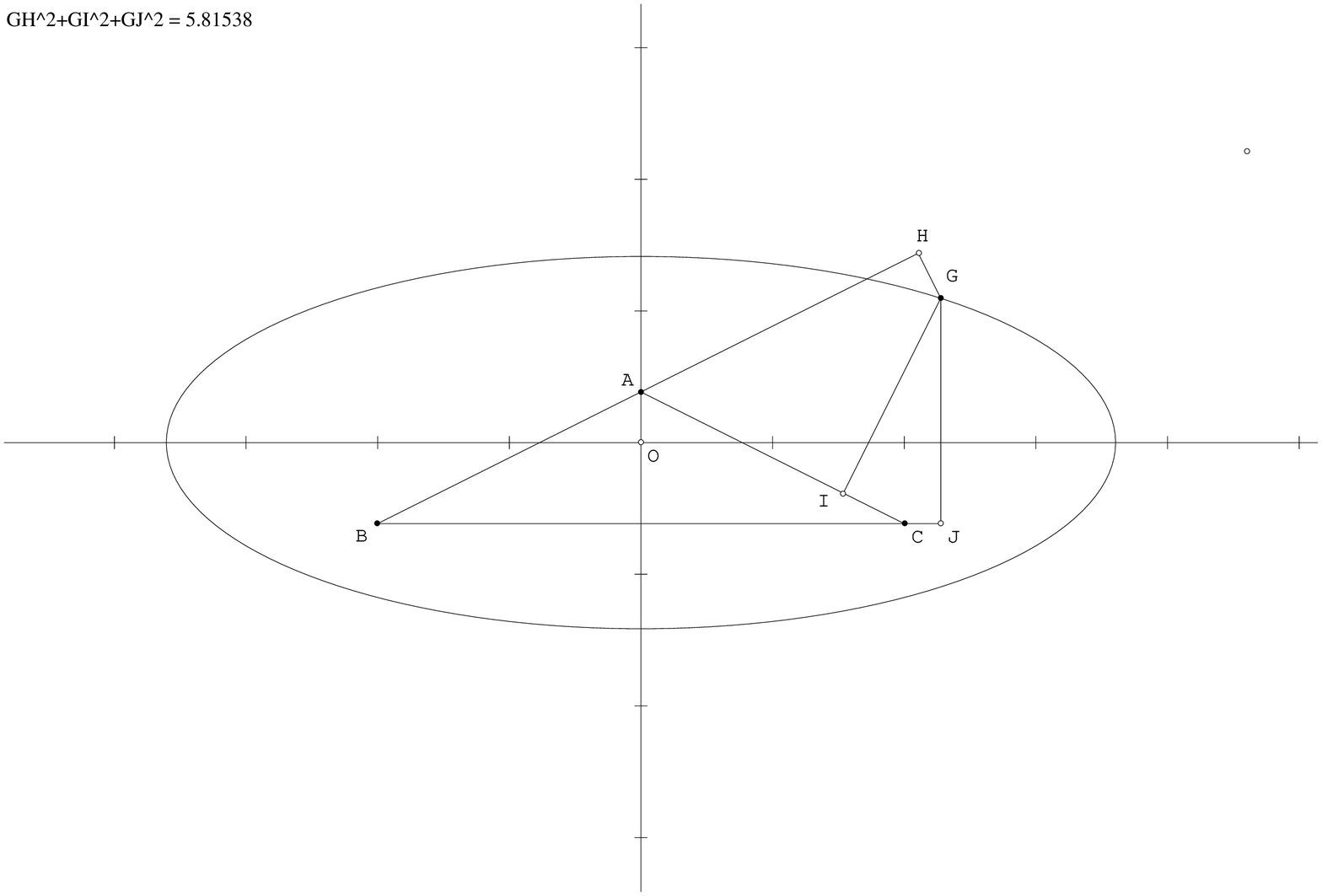}{\special{language
"Scientific Word";type "GRAPHIC";maintain-aspect-ratio TRUE;display
"USEDEF";valid_file "F";width 5.8435in;height 3.9539in;depth
0pt;original-width 10.6692in;original-height 7.2004in;cropleft "0";croptop
"1";cropright "1";cropbottom "0";filename '1-2.eps';file-properties "XNPEU";}%
}

\FRAME{ftbpFU}{5.9897in}{4.0542in}{0pt}{\Qcb{Sanghareeb $G$ moves along a
circle with $k=10$}}{\Qlb{fig7}}{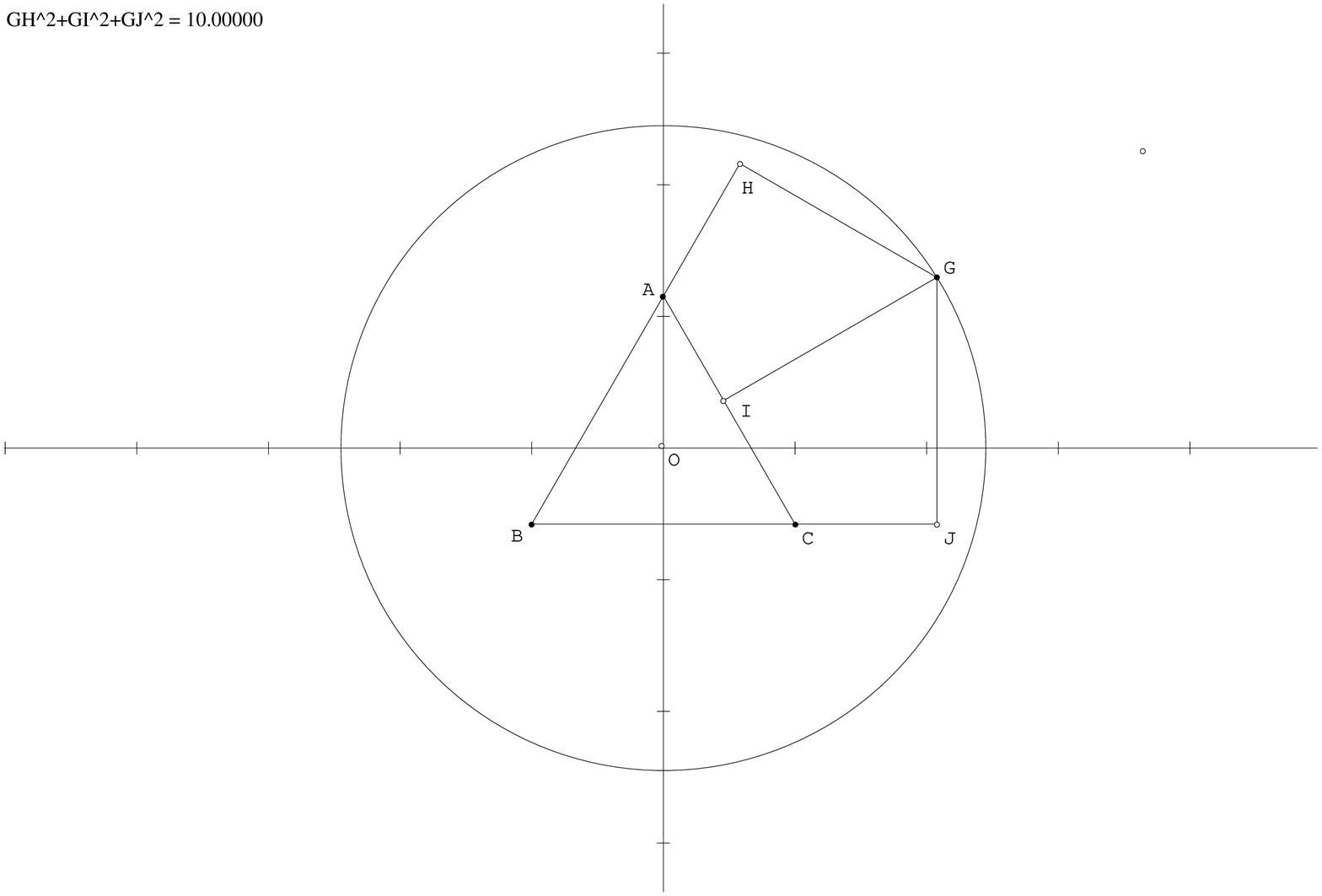}{\special{language "Scientific
Word";type "GRAPHIC";maintain-aspect-ratio TRUE;display "USEDEF";valid_file
"F";width 5.9897in;height 4.0542in;depth 0pt;original-width
10.6692in;original-height 7.2004in;cropleft "0";croptop "1";cropright
"1";cropbottom "0";filename 'circle.eps';file-properties "XNPEU";}}

\hspace{0in}

\end{document}